\documentclass[a4paper,10pt]{amsart}
\usepackage{amsmath,amssymb,amsfonts,amsthm}
\usepackage{latexsym}
\usepackage[all]{xy}
\usepackage{longtable}
\usepackage{supertabular}
\input xypic
\theoremstyle{plain}
\newtheorem{theorem}[subsection]{{\bf Theorem}}
\newtheorem*{theorem*}{{\bf Theorem}}
\newtheorem{corollary}[subsection]{{\bf Corollary}}
\newtheorem*{corollary*}{{\bf Corollary}}
\newtheorem{proposition}[subsection]{{\bf Proposition}}
\newtheorem{lemma}[subsection]{{\bf Lemma}}

\theoremstyle{definition}

\theoremstyle{remark}
\newtheorem{remark}[subsection]{{\it Remark}}
\newtheorem{example}[subsection]{{\it Example}}
\numberwithin{equation}{subsection}
\DeclareMathOperator{\im}{im}

\DeclareMathOperator{\HH}{H}
\DeclareMathOperator{\K}{K}
\DeclareMathOperator{\B}{B}
\DeclareMathOperator{\Bt}{\tilde{B}}
\DeclareMathOperator{\E}{E}
\DeclareMathOperator{\M}{M}
\DeclareMathOperator{\St}{St}

\DeclareMathOperator{\Hom}{Hom}
\DeclareMathOperator{\GL}{GL}

\DeclareMathOperator{\Frat}{Frat}
\DeclareMathOperator{\res}{res}
\DeclareMathOperator{\cor}{cor}
\DeclareMathOperator{\tra}{tra}

\newcommand{\QZ}{\mathbb{Q}/\mathbb{Z}}
\newcommand{\Z}{\mathbb{Z}}
\newcommand{\cwedge}{\curlywedge}
\newcommand{\dwedge}{\tilde{\wedge}}
\begin{document}
\title[Unramified Brauer groups]
{Unramified Brauer groups of finite and infinite groups}
\author{Primo\v z Moravec}
\address{{
Department of Mathematics \\
University of Ljubljana \\
Jadranska 21 \\
1000 Ljubljana \\
Slovenia}}
\email{primoz.moravec@fmf.uni-lj.si}
\dedicatory{In loving memory of my father}
\subjclass[2000]{}
\thanks{The author is indebted to
Ming-chang Kang and Boris Kunyavski\u\i \ for sending him their preprint \cite{Chu09}.
He would also like to thank the referees for their comments.}
\date{\today}
\begin{abstract}
\noindent
The Bogomolov multiplier is a group theoretical invariant isomorphic to
the unramified Brauer group of a given quotient space.
We derive a homological version of the Bogomolov multiplier, prove a Hopf-type formula, find a five term exact sequence corresponding
to this invariant, and describe the role of the Bogomolov multiplier in the theory of central extensions. A new description of the Bogomolov
multiplier of a nilpotent group of class two is obtained. 
We define the Bogomolov multiplier within K-theory and show that proving its triviality is equivalent
to solving a long-standing problem posed by Bass. An algorithm for computing the Bogomolov multiplier is developed.
\end{abstract}
\maketitle
\section{Introduction}
\label{s:intro}

\noindent
In this paper we develop a homological version of a group theoretical invariant that has served as one of the main tools in studying
the problem of stable rationality of quotient spaces.
Let $G$ be a finite group and $V$ a faithful representation of $G$ over $\mathbb{C}$. Then there is a natural action of $G$ upon
the field of rational functions $\mathbb{C}(V)$. A problem posed by Emmy Noether \cite{Noe16} asks as to whether 
the field of $G$-invariant functions $\mathbb{C}(V)^G$ is purely transcendental over $\mathbb{C}$, i.e.,
whether the quotient space $V/G$ is {\it rational}. A question related to the above mentioned
is whether $V/G$ is {\it stably rational}, that is,
whether there exist independent variables $x_1,\ldots ,x_r$ such that $\mathbb{C}(V)^G(x_1,\ldots ,x_r)$ becomes a pure
transcendental extension of $\mathbb{C}$. This problem has close connection
with L\"{u}roth's problem \cite{Saf91} and the inverse Galois problem \cite{Swa83,Sal84}.
By Hilbert's Theorem 90 stable rationality of $V/G$
does not depend upon the choice of $V$, but only on the group $G$. Saltman \cite{Sal84} found examples of groups $G$ such
that $V/G$ is not stably rational over $\mathbb{C}$. His main method was application of the unramified cohomology group
$\HH^2_{\rm nr}(\mathbb{C}(V)^G,\QZ )$ as an obstruction. 
A version of this invariant had been used before by Artin and Mumford \cite{Art72} who constructed unirational varieties over
$\mathbb{C}$ that were not rational.
Bogomolov \cite{Bog88} further explored this cohomology group. He proved 
that $\HH^2_{\rm nr}(\mathbb{C}(V)^G,\QZ )$ is canonically isomorphic to a certain subgroup $\B_0(G)$ (defined in Section
\ref{s:brauer}) of the {\it Schur multiplier} $\HH ^2(G,\QZ )$ of $G$. Kunyavski\u\i ~\cite{Kun08} coined the term the
{\it Bogomolov multiplier} of $G$ for the group $\B_0(G)$.
Bogomolov used the above description to find
new examples of groups with $\HH^2_{\rm nr}(\mathbb{C}(V)^G,\QZ )\neq 0$. Subsequently,
Bogomolov, Maciel and Petrov \cite{Bog04} showed that $\B_0(G)=0$ when $G$ is a finite simple group of Lie type $A_\ell$, whereas
Kunyavski\u\i \ \cite{Kun08} recently proved that $\B_0(G)=0$ for every quasisimple or almost simple group $G$. 
Bogomolov's conjecture
that $V/G$ is stably rational over $\mathbb{C}$ for every finite simple group $G$, nevertheless, still remains open.

\smallskip

We first observe
that if $G$ is a finite group, then $\B_0(G)$ is canonically isomorphic to $\Hom (\Bt_0(G),\QZ )$, where the group $\Bt_0(G)$ can be
described as a section of the {\it nonabelian exterior square} $G\wedge G$ of the group $G$. The latter appears implicitly in Miller's
work \cite{Mil52}, and was further developed by Brown and Loday \cite{Bro87}.
Let $\gamma _2(G)$ be the derived subgroup of $G$, and denote the kernel of the 
commutator homomorphism $G\wedge G\to\gamma _2(G)$ by $\M (G)$. 
Miller \cite{Mil52} proved that there is a natural isomorphism between $\M(G)$ and $\HH _2(G,\mathbb{Z})$. Using this description, we prove that
$\Bt_0(G)=\M(G)/\M_0(G)$, where $\M_0(G)$ is the subgroup of $\M (G)$ generated by all $x\wedge y$ such that $x,y\in G$ commute. In the finite
case, $\Bt_0(G)$ is thus (non-canonically) isomorphic to $\B_0(G)$.  The functor $\Bt_0$ can be studied within the category of all groups, and
this is the main goal of the paper. In the first part we prove a Hopf-type formula for $\Bt_0(G)$ 
by showing that if $G$ is given by a free presentation $G=F/R$, then
$$\Bt_0(G)\cong \frac{\gamma _2(F)\cap R}{\langle \K (F)\cap R\rangle},$$
where $\K (F)$ denotes the set of commutators in $F$. A special case of this was implicitly used before
by Bogomolov \cite{Bog88}, and Bogomolov, Maciel and Petrov \cite{Bog04}.
With the help of the above formula we derive a five term exact sequence
$$\Bt_0(G)\longrightarrow \Bt_0(G/N)\longrightarrow\frac{N}{\langle \K(G)\cap N]\rangle}\longrightarrow 
G^{\rm ab}\longrightarrow
(G/N)^{\rm ab}\longrightarrow 0,$$
where $G$ is any group and $N$ a normal subgroup of $G$. This is a direct analogue of the well-known five term homological sequence.
By applying Kunyavski\u\i 's work and the above sequence we obtain the following group theoretical result: 
If $G$ is a finite group and $S$ its solvable radical,
then $S\cap \gamma _2(G)=\langle S\cap \K(G)\rangle$. Furthermore, we compute $\Bt _0(G)$ when $G$ is a finite group that is a 
split extension. This corresponds to a well-known result of Tahara \cite{Tah72} who computed the Schur multiplier of semidirect
product of groups (see also \cite{Kar87}). In particular, we obtain a closed formula for $\B_0(G)$ when $G$ is a Frobenius group.

\smallskip

In his paper \cite{Bog88}, Bogomolov extended the definition of $\B_0(G)$ to cover all algebraic groups $G$, cf. Section \ref{s:brauer}. This
can be further extended in a natural way to cover all infinite groups. We prove here that if $G$ is any group, then $\B_0(G)$ is canonically
isomorphic to $\Hom (\Bt_{0\mathcal F}(G),\QZ )$, where $\Bt_{0\mathcal F}(G)$ is the quotient of the subgroup $\M_{\mathcal{F}}(G)$ of 
$\HH_2(G,\Z )$ generated by all images of corestriction maps $\cor_G^H:\HH _2(H,\Z )\to \HH_2(G,\Z)$, where $H$ runs through
all finite subgroups of $G$, by the subgroup $ \M_{0\mathcal{F}}(G)$ generated by all $\im\cor_G^A$, where $A$ runs through all finite abelian
subgroups of $G$. As a consequence we show that if $G$ is a locally finite group, then $\Bt_0(G)\cong\Bt_{0\mathcal{F}}(G)$. This in particular
applies to periodic linear groups. On the other hand, $\Bt_0(G)$ and $\Bt_{0\mathcal{F}}(G)$ may fail to be isomorphic in general.

\smallskip

One of the goals of the paper is to exhibit the role of $\Bt_0(G)$ in studying certain types of central extensions of $G$. This is motivated by the 
classical theory of Schur multipliers which are the cornerstones of the extension theory of groups.
We define $G\cwedge G$ to be the quotient of $G\wedge G$ by $\M_0(G)$. Then it is clear that the sequence
$\Bt_0(G)\rightarrowtail G\cwedge G\twoheadrightarrow \gamma _2(G)$ is exact, therefore $\Bt_0(G)$ can be thought of as the obstruction to
$G\cwedge G$ being isomorphic to $\gamma _2(G)$. This corresponds to a result of Miller \cite{Mil52} who
demonstrated that the nonabelian exterior square $G\wedge G$ of a group $G$ fits into the short exact sequence
$\M(G)\rightarrowtail G\wedge G\twoheadrightarrow \gamma _2(G)$.
This construction enables us to prove that if $G$ is a finite group, then for every stem extension $(E,\pi ,A)$ producing $\Bt_0(G)$ we have that $\gamma _2(E)$ and $G\cwedge G$ are of the same order. Furthermore, there exists
a stem extension of this kind such that $\gamma _2(E)$ is actually isomorphic to $G\cwedge G$. This can be seen as a direct analogue
of the well-known fact that if $G$ is finite, then $G\wedge G$ is naturally isomorphic to the derived subgroup of an arbitrary covering group of $G$.
In addition to that, we prove that if $G$ is a perfect
group, then $G\cwedge G$ is universal within the class of central extensions $E$ of $G$ with the property that every commuting pair of elements
in $G$ has commuting lifts in $E$. Again, this corresponds to the fact that if 
$G$ is a perfect group, then $G\wedge G$ is the universal central extension of $G$.

The first known examples of finite groups $G$ with $\B_0(G)\neq 0$ were found among $p$-groups of class 2 \cite{Bog88,Sal84}. Bogomolov
obtained a description of $\B_0(G)$ when $G$ is a $p$-group of class 2 with $G^{\rm ab}$ elementary abelian. Here we obtain a description
of $\Bt_0(G)$ for any group $G$ that is nilpotent of class 2. More precisely, we show that 
$\Bt_0(G)\cong \ker (\HH_2(G^{\rm ab},\Z )\to\gamma _2(G))/\ker (\HH_2(G^{\rm ab},\Z )\to G\cwedge G)$. In the case when 
$G$ is a $p$-group of class 2 with $G^{\rm ab}$ elementary abelian, this can be further refined using the Blackburn-Evens theory
\cite{Bla79}.

\smallskip

The functor $\Bt_0$ has applications in K-theory. For a unital ring $\Lambda$ define $\Bt_0\Lambda =\Bt_0(\E (\Lambda ))$ where
$\E (\Lambda )$ is the subgroup of $\GL (\Lambda )$ generated by elementary matrices. We prove that $\Bt_0\Lambda$ is naturally
isomorphic to $\K_2\Lambda /\langle \K (\St (\Lambda )\cap \K_2\Lambda\rangle$, where $\St (\Lambda )$ is the Steinberg group. This is
related to a conjecture posed by Bass \cite[Problem 3]{Den73} that $\K_2\Lambda$ is always generated by the so-called Milnor elements.
We show that this problem has a positive solution for a ring $\Lambda$ if and only if $\Bt_0\Lambda$ is trivial. The latter is for instance true
for commutative semilocal rings. A possible approach towards solving Bass' problem could be based on the result that $\Bt_0\Lambda$ is naturally
isomorphic to $\Bt_0(\GL (\Lambda ))$. 

\smallskip

In general it is hard to compute $\B_0(G)$, due to its cohomological description. 
Chu, Hu, Kang, and Kunyavski\u\i ~\cite{Chu09} recently completed calculations of $\B_0(G)$ for all groups of order $\le 64$.
The homological nature of $\Bt_0$, on the other hand, allows machine computation of $\Bt_0(G)$ for polycyclic groups $G$.
There is an efficient algorithm developed recently by Eick and Nickel \cite{Eic08} for computing $G\wedge G$ in case $G$ is polycyclic.
Based on that we develop and implement an algorithm for computing $\Bt_0(G)$ for finite solvable groups $G$. We use this
algorithm to determine the Bogomolov multiplier of all solvable groups of order $\le 729$, apart 
from the orders 512, 576 and 640. Our computations in particular
show that there exist  three groups of order 243 with nontrivial unramified Brauer group. This contradicts a result of
Bogomolov \cite{Bog88} claiming that if $G$ is a finite $p$-group of order at most $p^5$, then $\B_0(G)=0$.
\section{Preliminaries and notations}
\label{s:prelim}

\noindent
In this section we fix some notations used throughout the paper.
Let $G$ be a group and $x,y\in G$.  We use the notation ${}^xy=xyx^{-1}$ for conjugation from the left. 
The commutator $[x,y]$ of elements $x$ and $y$ is defined by
$[x,y]=xyx^{-1}y^{-1}={}^xyy^{-1}$. If $H$ and $K$ are subgroups of $G$, then we
define $[H,K]=\langle [h,k]\mid h\in H, k\in K\rangle$. The commutator subgroup $\gamma _2(G)$ of $G$ is defined to
be the group $[G,G]$. The {\it set} 
$\{ [x,y]\mid x,y\in G \}$ of all commutators of $G$ is denoted by $\K(G)$.

We recall the definition and basic properties of the nonabelian exterior product of groups. The reader is referred to \cite{Bro87,Mil52} for more
thorough accounts on the theory and its generalizations.
Let $G$ be a group and $M$ and $N$ normal subgroups of $G$. We form the group $M\wedge N$, generated by the symbols
$m\wedge n$, where $m\in M$ and $n\in N$, subject to the following relations:
\begin{align}
\label{eq:tens1}
mm'\wedge n &= ({}^mm'\wedge {}^mn)(m\wedge n),\\
\label{eq:tens2}
m \wedge nn' &= (m\wedge n)({}^nm\wedge {}^nn'),\\
\label{eq:tens3}
x\wedge x &= 1,
\end{align}
for all $m,m'\in M$, $n,n'\in N$ and $x\in M\cap N$.

Let $L$ be a group. A function $\phi : M \times N \to L$ is called a {\it crossed
pairing} if for all $m, m' \in M$, $n, n' \in N$,
$\phi (mm', n) = \phi ({}^mm', {}^mn)\phi (m,n)$,
$\phi (m,nn') = \phi (m,n)\phi ({}^nm, {}^nn')$, and $\phi (x,x)=1$ for all $x\in M\cap N$.
A crossed pairing $\phi$ determines a unique homomorphism of groups 
$\phi ^*: M\wedge N\to L$ such that $\phi ^*(m\wedge n)=\phi (m,n)$
for all $m\in M$, $n\in N$.

The group $G\wedge G$ is said to be
the {\it nonabelian exterior square} of $G$.
By definition, the commutator map $\kappa :G\wedge G\to \gamma _2(G)$, given by
$g\wedge h\mapsto [g,h]$, is a well defined homomorphism of groups.
Clearly $\M(G)=\ker\kappa$ is central in $G\wedge G$, and
$G$ acts trivially via diagonal action on  $\M(G)$.
Miller \cite{Mil52} proved that there is a natural isomorphism between $\M(G)$ and $\HH _2(G,\mathbb{Z})$. A direct consequence
of this result is that if a group $G$ is given by a free presentation $G\cong F/R$, then $G\wedge G$ is naturally isomorphic to
$\gamma _2(F)/[R,F]$.

The following lemma collects some basic identities that hold in the nonabelian exterior square of a group:

\begin{lemma}[\cite{Bro87}]
\label{l:ext}
Let $G$ be a group and $x,y,z,w\in G$.
\begin{enumerate}
\item[(a)] $x\wedge y=(y\wedge x)^{-1}$.
\item[(b)] ${}^{x^{-1}}(x\wedge y)=y\wedge x^{-1}$.
\item[(c)] ${}^{[z,w]}(x\wedge y)=(z\wedge w)(x\wedge y)(z\wedge w)^{-1}$.
\end{enumerate}
\end{lemma}
\section{The unramified Brauer group}
\label{s:brauer}

\noindent
Let $G$ be a finite group and $V$ a faithful representation of $G$ over $\mathbb{C}$.
Bogomolov \cite{Bog88} proved that the unramified Brauer group
$\HH^2_{\rm nr}(\mathbb{C}(V)^G,\QZ )$ is canonically isomorphic to the group
\begin{equation}
\label{eq:b0def}
\B_0(G)=\bigcap _{\tiny\begin{matrix}A\le G,\\ A \hbox{ abelian}\end{matrix}} \ker \res ^G_A,
\end{equation}
where $\res ^G_A:\HH^2(G,\QZ )\to \HH^2(A,\QZ )$ is the usual cohomological restriction map.
Our first aim is to obtain a homological description of $\B_0(G)$. Thus
we need a dual of the above construction.
Let $H$ be a subgroup of $G$. Then there is a corestriction map $\cor _G^H: \HH_2(H,\Z ) \to \HH_2(G,\Z )$.
On the other hand, we have a natural map $\tau _G^H:H\wedge H\to G\wedge G$. 
Identifying $\HH_2(G,\Z )$ with $\M(G)$ and
$\HH_2(H,\Z )$ with $\M(H)$, we can write $\cor _G^H=\tau _G^H|_{\M(H)}$. Thus we have the following commutative diagram with
exact rows:
\begin{equation}
\label{eq:diagram}
\xymatrix{
0\ar[r] & \HH_2(H,\Z )\ar[r]\ar[d]_{\cor _G^H} & H\wedge H \ar[r]\ar[d]_{\tau _G^H} & \gamma _2(H)\ar[r]\ar[d] & 1\\
0\ar[r] & \HH_2(G,\Z )\ar[r] & G\wedge G\ar[r] & \gamma _2(G)\ar[r] & 1
}
\end{equation}

Now define
$$\M_0(G)=\langle \cor _G^A \M(A) \mid A\le G, A \hbox{ abelian}\rangle .$$
This group can be described as a subgroup of $G\wedge G$ in the following way.

\begin{lemma}
\label{l:A0}
Let $G$ be a group. Then
$$\M_0(G)=\langle x\wedge y\mid x,y\in G, [x,y]=1\rangle .$$
\end{lemma}
\begin{proof}
Denote $N=\langle x\wedge y\mid x,y\in G, [x,y]=1\rangle .$
Suppose that $x,y\in G$ commute. Then $A=\langle x,y\rangle$ is an abelian subgroup of $G$, hence
$\cor _G^A \M(A)\le \M_0(G)$. In particular, $x\wedge y\in \M_0(G)$.

Conversely, let $A$ be an abelian subgroup of $G$. Let $w\in \cor _G^A \M(A)$. Then $w$ can be written as
$$w=\prod _{i=1}^r(a_i\wedge b_i),$$
where $a_i,b_i\in A$. Since $[a_i,b_i]=1$ for all $i=1,\ldots ,r$, it follows that
$w\in N$. This concludes the proof.
\end{proof}

For a group $G$ denote 
$$\Bt _0(G)=\M(G)/\M_0(G).$$
With this notation we have the following result.

\begin{theorem}
\label{t:main}
Let $G$ be a finite group. Then $\B_0(G)$ is naturally isomorphic to $\Hom (\Bt _0(G),\QZ )$, and thus
$\B_0(G)\cong \Bt _0(G)$ (non-canonically).
\end{theorem}
\begin{proof}
At first we describe the natural isomorphism
between the Schur multiplier $\HH^2(G,\QZ )$ and $\Hom (\HH_2(G,\Z) ,\QZ )$ in terms of the nonabelian exterior square of $G$.
Choose $\gamma \in \HH^2(G,\QZ )$ and let
$$\xymatrix{ 0\ar[r] & \QZ\ar[r]^i & G_\gamma\ar[r]^\pi & G\ar[r] & 1}$$
be the central extension associated to $\gamma$.
Define a map $G\times G\to \gamma _2(G_\gamma )$ by the rule $(x,y)\mapsto [\bar{x},\bar{y}]$, where $\bar{x}$ and
$\bar{y}$ are preimages in $G_\gamma$ under $\pi$ of $x$ and $y$, respectively. This map is well defined. Furthermore, it is a crossed
pairing, hence it induces a homomorphism $\lambda _\gamma :G\wedge G\to\gamma _2(G_\gamma )$ given by
$\lambda _\gamma (x\wedge y)=[\bar{x},\bar{y}]$ for $x,y\in G$. It is clear that if $c\in \M(G)$, then
$\lambda _\gamma (c)\in i(\QZ)$, therefore the restriction of $\lambda _\gamma$ to $\M(G)$
(still denoted by $\lambda _\gamma)$
belongs to $\Hom (\M(G),\QZ )$. The map $\Theta :\HH^2(G,\QZ )\to \Hom (\M(G),\QZ )$ given by
$\gamma\mapsto\lambda _\gamma$ is a homomorphism of groups.

Conversely, let $\varphi\in \Hom (\M(G),\QZ )$. Let $H$ be a covering group of $G$. In other words, we have a central extension
$$\xymatrix{ 0\ar[r] & Z\ar[r]^j & H\ar[r]^\rho & G\ar[r] & 1}$$
with $jZ\le \gamma _2(H)$ and $Z\cong \M(G)$.
Every finite group
has at least one covering group by a result of Schur, cf. \cite[Hauptsatz V.23.5]{Hup67}. 
By \cite{Bro87} we have that $\gamma _2(H)$ is canonically
isomorphic to $G\wedge G$.
Upon identifying $\gamma _2(H)$ with $G\wedge G$, we may assume
without loss of generality that $\M(G)$ is a subgroup of $H$. Choose a section $\mu :G\to H$ of $\rho$ and define
a map $f:G\times G\to H$ by $f(x,y)=\mu (x)\mu (y)\mu (xy)^{-1}$ for $x,y\in G$. It is straightforward to verify $f$ maps $G\times G$ into $\M(G)$,
and that $\varphi f\in Z^2(G,\QZ )$. The cohomology class of $\varphi f$ does not depend upon the choice of $\mu$. We therefore have
a map (the so-called transgression map)
$$\tra : \Hom(\M(G),\QZ )\to \HH^2(G,\QZ )$$
given by $\tra (\varphi )=[\varphi f]$. This is easily seen to be a homomorphism, and $\Theta$ is its inverse.

Now choose $\gamma\in \B_0(G)$ and let the map $\Theta :\HH^2(G,\QZ )\to \Hom (\M(G),\QZ )$ be defined as above. Denote
$\lambda _\gamma =\Theta (\gamma )$. Let $x,y\in G$ and suppose that $[x,y]=1$. Then $A=\langle x,y\rangle$ is an
abelian subgroup of $G$, therefore $\res _A^G(\gamma )=0$. This implies that
$\lambda _\gamma (x\wedge y)=[\bar{x},\bar{y}]=1$. Therefore $\Theta$ induces a 
homomorphism $\tilde{\Theta} : \B_0(G)\to \Hom (\M(G)/\M_0(G),\QZ )$.

Let $\varphi\in \Hom (\M(G)/\M_0(G),\QZ )$. Then $\varphi$ can be lifted to a homomorphism $\bar{\varphi} :\M(G)\to\QZ$.
Put $\gamma =\tra (\bar{\varphi})$.
Suppose that
$$\xymatrix{ 0\ar[r] & \QZ\ar[r]^i & G_\gamma\ar[r]^\pi & G\ar[r] & 1}$$
is a central extension associated to $\gamma$. Choose an arbitrary
bicyclic subgroup $A=\langle a,b\rangle$ of $G$. Then we have a central extension
$$\xymatrix{ 0\ar[r] & \QZ\ar[r]^i & A_\gamma\ar[r]^{\pi |_{A_\gamma}} & A\ar[r] & 1}$$
that corresponds to $\res _A^G(\gamma )$. Since $[a,b]=1$, we have that $a\wedge b\in \M_0(G)\le\ker\bar{\varphi}$, therefore
$[\bar{a},\bar{b}]=1$ in $A_\gamma$. It follows that $A_\gamma$ is abelian, thus
$\gamma\in \B_0(G)$. Hence the transgression map induces a homomorphism $\widetilde{\tra} :  \Hom (\M(G)/\M_0(G),\QZ )\to \B_0(G)$
whose inverse is $\tilde{\Theta}$.
\end{proof}

The definition of $\B_0(G)$ can be extended to infinite groups as follows \cite{Bog88}. Let $G$ be a group. Define
$$K_G=\{\gamma\in \HH ^2(G,\QZ )\mid \res ^G_H\gamma =0\hbox{ for every finite } H\le G\}.$$
Let $\B_0(G)$ be the subgroup of $\HH ^2(G,\QZ )/K_G$ consisting of all $\gamma +K_G$ with the property
that $\res _A^G\gamma =0$ for every finite abelian subgroup $A$ of $G$. 
It is clear that if $G$ is a finite group, then this definition of $\B_0(G)$ coincides with the one given by \eqref{eq:b0def}.
Bogomolov \cite[Theorem 3.1]{Bog88} showed that if $G$ is an algebraic group, then $\B_0(G)$ is isomorphic to
$\HH^2_{\rm nr}(\mathbb{C}(V)^G,\QZ )$, where $V$ is any generically free representation of $G$.

In order to obtain a homological description of $\B_0(G)$ for infinite groups, we denote
$$\M_{\mathcal{F}}(G)=\langle \cor ^H_G\M (H)\mid H\le G, |H|<\infty\rangle$$
and
$$\M_{0\mathcal{F}}(G)=\langle \cor ^A_G\M (A)\mid A\le G, |A|<\infty, A\hbox{ abelian}\rangle .$$
Note that a similar argument as that of Lemma \ref{l:A0} shows
that $\M_{0\mathcal{F}}(G)=\langle x\wedge y\mid [x,y]=1, |x|<\infty ,|y|<\infty\rangle$.
Now define $\Bt _{0\mathcal{F}}(G)=\M_{\mathcal{F}}(G)/\M_{0\mathcal{F}}(G)$. Then we have:

\begin{theorem}
\label{p:b0finitary}
Let $G$ be a group. Then the group $\B_0(G)$ is naturally isomorphic to
$\Hom (\Bt _{0\mathcal{F}}(G),\QZ )$.
\end{theorem}
\begin{proof}
We have a natural isomorphism $\HH ^2(G,\QZ)\cong \Hom (\M(G),\QZ )$. For $\gamma\in\HH ^2(G,\QZ)$ denote
by $\lambda_\gamma$ the corresponding element of $\Hom (\M(G),\QZ )$. By our definition we have that
$\gamma\in K_G$ if and only if $\M _{\mathcal{F}}(G)\le \ker\lambda _\gamma$. Therefore
$\HH ^2(G,\QZ)/K_G$ is naturally isomorphic to $\Hom (\M_{\mathcal{F}}(G),\QZ )$. Adapting the argument of the proof
of Theorem \ref{t:main}, we obtain the required result.
\end{proof}

In general, $\Bt _0(G)$ and $\Bt _{0\mathcal{F}}(G)$ may be quite different. For example, if $G$ is a one-relator group
with torsion, then $G$ has a presentation $G=\langle X\mid s^m\rangle$, where $s$ is not a proper power in the free group
over $X$. By a result of Newman \cite{New68}, every finite subgroup of $G$ is conjugate to a subgroup of $\langle s\rangle$,
hence
$\Bt _{0\mathcal{F}}(G)=0$. On the other hand,
all centralizers of nontrivial elements of $G$ are cyclic \cite{New68}, hence $\M_0(G)=0$ and therefore
$\Bt_0(G)\cong \HH_2(G,\mathbb{Z})$. The latter can be nontrivial, cf. Lyndon \cite{Lyn50}.

In the case of locally finite groups we have the following:

\begin{corollary}
\label{c:locallyfin}
Let $G$ be a locally finite group. Then $\Bt _0(G)\cong\Bt _{0\mathcal{F}}(G)$.
\end{corollary}
\begin{proof}
Every group $G$ is a direct limit of its finitely generated subgroups $\{ G_\lambda\mid \lambda\in\Lambda\}$. If $G$ is locally finite,
then the groups $G_\lambda$ are all finite. Since $\M(G)\cong\varinjlim \M(G_\lambda)$, we conclude that $\M (G)=\M_{\mathcal{F}}(G)$.
Since $G$ is periodic, we also have that $\M_{0\mathcal{F}}(G)=\M _0(G)$, hence the result.
\end{proof}

Corollary \ref{c:locallyfin} applies, for example, to periodic linear groups.
On the other hand, there exist finitely generated periodic groups $G$ (even of finite exponent) such that
$\Bt_{0\mathcal{F}}(G)=0$, yet $\Bt_0(G)$ is nontrivial.

\begin{example}
\label{ex:burnside}
Suppose $m>1$ and let $n>2^{48}$ be odd. Let $F$ be a free group of rank $m$. Denote
$B(m,n)=F/F^n$, the {\it free Burnside group} of rank $m$ and exponent $n$. Ivanov \cite{Iva94} showed
that all centralizers of nontrivial elements of $B(m,n)$ are cyclic, and that every finite subgroup of $B(m,n)$
is cyclic. From here it follows that $\Bt _{0\mathcal{F}}(B(m,n))=0$ and $\Bt_0(B(m,n))\cong\HH _2(B(m,n),\mathbb{Z})$. The latter
group is free abelian of countable rank, cf. \cite[Corollary 31.2]{Ols91}.
\end{example}

In the rest of the paper we mainly consider
the properties of $\Bt_0(G)$. Obviously $\Bt_0$ is a covariant functor
from $\mathbf{Gr}$ to $\mathbf{Ab}$. It is well known that the homology functor commutes with direct limits. It turns out
that $\Bt_0$ enjoys the same property:

\begin{proposition}
\label{p:dirlim}
The functor $\Bt_0$ commutes with direct limits. More precisely, if $\{ G_\lambda ,\alpha_\lambda^\mu\mid \lambda\le\mu\in\Lambda\}$
is a direct system of groups and $G$ its direct limit, then $\Bt_0(G)$ is the direct limit of
 $\{ \Bt_0(G_\lambda ),\Bt_0(\alpha_\lambda^\mu)\mid \lambda\le\mu\in\Lambda\}$.
\end{proposition}
\begin{proof}
For every $\lambda\in\Lambda$ we have
$$0\longrightarrow \M_0(G_\lambda )\longrightarrow\M (G_\lambda )\longrightarrow \Bt_0(G_\lambda )\longrightarrow 0,$$
hence the diagram
$$
\xymatrix{
0\ar[r] & \varinjlim  \M_0(G_\lambda )\ar[r]\ar[d]_{\alpha '} & \varinjlim \M (G_\lambda ) \ar[r]\ar[d]_{\alpha} & \varinjlim\Bt_0(G_\lambda )
 \ar[r]\ar[d]_{\tilde{\alpha}} & 0\\
0\ar[r] & \M_0(G)\ar[r] & \M(G)\ar[r] & \Bt_0(G)\ar[r] & 0
}
$$
is commutative with exact rows. Here $\alpha$ is the natural isomorphism, and $\alpha '$ is its restriction. Clearly
$\alpha '$ is an isomorphism, hence so is $\tilde{\alpha}$.
\end{proof}

\begin{proposition}
\label{p:freeprod}
Let $G_1$ and $G_2$ be groups. Then $\Bt _0(G_1\ast G_2)\cong \Bt_0(G_1)\times \Bt_0(G_2)$
and $\Bt _{0\mathcal{F}}(G_1\ast G_2)\cong \Bt_{0\mathcal{F}}(G_1)\times \Bt_{0\mathcal{F}}(G_2)$.
\end{proposition}
\begin{proof}
Let $G=G_1\ast G_2$ and let $\iota _1:G_1\to G$ and $\iota _2:G_2\to G$ be the canonical injections.
Then the induced maps $\iota^*_i:\M(G_i)\to \M(G)$ are injective ($i=1,2$), $\iota^*_1\M(G_1)\cap \iota^*_2\M(G_2)=1$ and
$\M (G)=\iota^*_1\M(G_1)\times\iota^*_2\M(G_2)$ by \cite{Mil52}. Now let $a,b\in G\setminus\{ 1\}$ with $[a,b]=1$. By 
\cite[p. 196]{Mag76} we have the following possibilities. If $a\in {}^h\iota _1(G_1)$, then $b\in C_G(a)\le {}^h\iota _1(G_1)$, 
hence we
can write $a={}^h\iota _1(x)$ and $b={}^h\iota _1(y)$  for some commuting elements $x,y\in G_1$. In this case we get
$a\wedge b={}^h(\iota _1(x)\wedge \iota _1(y))=\iota _1(x)\wedge \iota _1(y)$, as $G$ acts trivially on $\M (G)$.
For $a\in {}^h\iota _2(G_2)$, the situation is similar.
If neither $a\in {}^h\iota _1(G_1)$ nor  $a\in {}^h\iota _2(G_2)$, $C_G(a)$ is infinite cyclic. In this case we clearly have that $a\wedge b=1$.
Therefore we conclude that
$\M_0(G)=\iota ^*_1\M_0(G_1)\times \iota ^*_2\M_0(G_2)$.
It follows from here that
$$\Bt _0(G)\cong \iota^\sharp_1\Bt _0(G_1)\times  \iota^\sharp_2\Bt _0(G_2),$$
where $\iota ^\sharp _i:\Bt _0(G_1)\to\Bt _0(G)$ are the maps induced by $\iota _i$, $i=1,2$. From the diagram
$$
\xymatrix{
1\ar[r] & \M_0(G_i )\ar[r]\ar[d]_{\iota_i ^*} & \M(G_i) \ar[r]\ar[d]_{\iota _i^*} & \Bt _0(G_i)
 \ar[r]\ar[d]_{\iota _i^\sharp} & 1\\
1\ar[r] & \M_0(G)\ar[r] & \M(G)\ar[r] & \Bt_0(G)\ar[r] & 1
}
$$
we see that $\iota _i^\sharp$, $i=1,2$, are both injective, therefore $\Bt _0(G_1\ast G_2)\cong \Bt_0(G_1)\times \Bt_0(G_2)$.

It remains to prove the corresponding assertion for  $\Bt _{0\mathcal{F}}(G_1\ast G_2)$. The above argument shows that
$\M _{0\mathcal{F}}(G)=\iota _1^*\M _{0\mathcal{F}}(G_1)\times \iota _2^*\M _{0\mathcal{F}}(G_2)$. By \cite[p. 54]{Bro82},
every finite subgroup of $G$ is conjugate to a subgroup of $G_1$ or $G_2$. Since $G$ acts trivially on $\M (G)$, we therefore conclude
that
$\M _{\mathcal{F}}(G)=\langle \cor ^H_G\M (H)\mid H\le G_1\hbox{ or } H\le G_2, |H|<\infty\rangle
=\iota _1^*\M _{\mathcal{F}}(G_1)\times \iota _2^*\M _{\mathcal{F}}(G_2)$. From here the result follows along the same lines as above.
\end{proof}

Let the group $G$ be given by a free presentation $G=F/R$, where $F$ is a free group and $R$ a normal subgroup of $F$.
By the well known Hopf formula \cite[Theorem II.5.3]{Bro82}
we have that $\M(G)\cong (\gamma _2(F)\cap R)/[R,F]$. The isomorphism is induced by the canonical
isomorphism $G\wedge G\to \gamma _2(F)/[R,F]$ given by $xR\wedge yR\mapsto [x,y][R,F]$. Under this map, $\M_0(G)$ can be identified
with the subgroup of $F/[F,R]$ generated by all the commutators in $F/[F,R]$ that belong to the Schur multiplier of $G$. In other words,
we have that $\M_0(G)\cong \langle \K(F/[R,F])\cap R/[R,F]\rangle =\langle \K(F)\cap R\rangle [R,F]/[R,F]=
\langle \K(F)\cap R\rangle/[R,F]$. Thus we have proved the following
Hopf-type formula for $\Bt_0(G)$:

\begin{proposition}
\label{p:hopf}
Let $G$ be a group given by a free presentation $G=F/R$. Then
$$\Bt_0(G)\cong \frac{\gamma _2(F)\cap R}{\langle \K(F)\cap R\rangle }.$$
\end{proposition}

This formula enables, in principle, explicit calculations of $\Bt_0(G)$, given a free presentation of $G$. For example, a word $w$ in a
free group $F$ is said to be a {\it commutator word} if $w=[u,v]$ for some $u,v\in F$. We have the following result:

\begin{corollary}
\label{c:relfree}
Let $\mathfrak{V}$ be a variety of groups defined by a commutator word $w$. If $G$ is a $\mathfrak{V}$-relatively free group,
then $\Bt _0(G)=0$.
\end{corollary}
\begin{proof}
Let $w$ be an $n$-variable commutator word.
$G$ can be presented as a quotient $F/\mathfrak{V}(F)$ of a free group $F$ by the verbal subgroup
$\mathfrak{V}(F)=\langle w(f_1,\ldots ,f_n)\mid f_1,\ldots ,f_n\in F\rangle$ of $F$.
Note that $\mathfrak{V}(F)\le\gamma _2(F)$ and $\langle \K (F)\cap\mathfrak{V}(F)\rangle =\mathfrak{V}(F)$. By
Proposition \ref{p:hopf} we get the result.
\end{proof}

On the other hand, there exist relatively free groups $G$ with $\Bt _0(G)\neq 0$, cf. Example \ref{ex:burnside} and Section \ref{s:comput}.

Another interpretation of $\Bt_0(G)$ for finite groups $G$ can be obtained via covering groups. 
Covering groups of a given group $G=F/R$ may not be unique,
yet their derived subgroups are all naturally isomorphic to $\gamma _2(F)/[R,F]$. Under this identification 
we have the following result.

\begin{proposition}
\label{p:cover}
Let $G$ be a finite group and $H$ its covering group. Let $Z$ be a central subgroup of $H$ such that
$Z\le \gamma _2(H)$, $Z\cong \M(G)$ and $H/Z\cong G$. Then
$$\Bt_0(G)\cong \frac{Z}{\langle \K(H)\cap Z\rangle}.$$
In particular, $\Bt_0(G)=0$ if and only if every element of $Z$ can be represented as a product of commutators
that all belong to $Z$.
\end{proposition}

We note here that a special case of Proposition \ref{p:cover} formed one of the crucial steps in proving the main
results of \cite{Bog04} and \cite{Kun08}.

One of the main features of the
homological description of $\B_0(G)$ is a five term exact sequence associated to the short exact sequence
$1\rightarrow N\rightarrow G\rightarrow G/N\rightarrow 1$ of groups. This sequence is an unramified Brauer group
analogue of the well known five term homology sequence, cf \cite[p. 46]{Bro82}.

\begin{theorem}
\label{t:5term}
Let $G$ be a group and $N$ a normal subgroup of $G$. Then we have the following exact sequence:
$$\Bt_0(G)\longrightarrow \Bt_0(G/N)\longrightarrow\frac{N}{\langle \K(G)\cap N]\rangle}\longrightarrow 
G^{\rm ab}\longrightarrow
(G/N)^{\rm ab}\longrightarrow 0.$$
\end{theorem}
\begin{proof}
Let $G$ have a free presentation $G=F/R$, and let $SR/R$ be the corresponding free presentation of $N$. 
Then Proposition
\ref{p:hopf} implies that $\Bt_0(G)\cong (\gamma _2(F)\cap R)/\langle \K(F)\cap R\rangle$ and 
$\Bt_0(G/N)\cong (\gamma _2(F)\cap RS)/\langle \K(F)\cap RS\rangle$. 
The canonical epimorhism $\rho :G\to G/N$ induces
a homomorphism $\rho ^\sharp :\Bt_0(G)\to \Bt_0(G/N)$. From the above Hopf formulae it follows that
$$\ker\rho ^\sharp =\frac{R\cap \langle \K(F)\cap RS\rangle}{\langle \K(F)\cap R\rangle}$$
and
$$\im\rho ^\sharp = \frac{\gamma _2(F)\cap \langle \K(F)\cap RS\rangle R}{\langle \K(F)\cap RS\rangle}.$$
It is straightforward to verify that $\langle \K(G)\cap N\rangle =\langle \K(F)\cap RS\rangle R/R$, therefore
$N/\langle \K(G)\cap N\rangle\cong RS/\langle \K(F)\cap RS\rangle R$. Thus there is a natural map
$\sigma :\Bt_0(G/N)\to N/\langle \K(G)\cap N\rangle$. We have that $\ker\sigma =\im\rho ^\sharp$ and
$$\im\sigma =\frac{(\gamma _2(F)\cap RS)R}{\langle \K(F)\cap RS\rangle R}=
\frac{\gamma _2(F)R\cap RS}{\langle \K(F)\cap RS\rangle R}=\frac{\gamma _2(G)\cap N}
{\langle \K(G)\cap N\rangle}.$$
Furthermore, there is a natural map $\pi :N/\langle \K(G)\cap N\rangle\to G^{\rm ab}$ 
whose kernel is equal to $\im\sigma$, and
$\im \pi =N\gamma _2(G)/\gamma _2(G)$. Finally, there 
is a surjective homomorphism $G^{\rm ab}\to (G/N)^{\rm ab}$
whose kernel is equal to $\im\pi$. From here our assertion readily follows.
\end{proof}

The proof of Theorem \ref{t:5term} also yields another exact sequence that is 
an analogue of the corresponding sequence
for Schur multipliers obtained by Blackburn and Evens \cite{Bla79}. More precisely, we have:

\begin{proposition}
\label{p:blaevens}
Let $G$ be a group given by a free presentation $G=F/R$ and let $N=SR/R$ be a normal subgroup of $G$. 
Then the sequence
$$0\rightarrow\frac{R\cap \langle \K(F)\cap RS\rangle}{\langle \K(F)\cap R\rangle}\rightarrow
\Bt_0(G)\rightarrow \Bt_0(G/N)\rightarrow \frac{N\cap\gamma _2(G)}{\langle \K(G)\cap N]\rangle}\rightarrow 0$$
is exact.
\end{proposition}

The above result
has the following group theoretical consequence:

\begin{corollary}
\label{c:commfitt}
Let $G$ be a finite group and $S$ the solvable radical of $G$, i.e., the largest solvable normal subgroup of $G$. 
Then $S\cap \gamma _2(G)=\langle S\cap \K(G)\rangle$.
\end{corollary}
\begin{proof}
The factor group $G/S$  does not contain proper nontrivial abelian normal subgroups, i.e., it is semisimple.
By a result of  Kunyavski\u\i \ \cite{Kun08} we
conclude that $\Bt_0(G/S)=0$. From Proposition \ref{p:blaevens} we get the desired result.
\end{proof}
\section{The `commutativity-preserving' nonabelian exterior product of groups}
\label{s:wedge}

\noindent
The nonabelian exterior square of a group encodes crucial information on the Schur multiplier of the group. In this section we 
introduce a related construction that plays a similar role when considering the functor $\Bt_0$.

Let $G$ be a group and $M$ and $N$ normal subgroups of $G$. We form the group $M\cwedge N$, generated by the symbols
$m\cwedge n$, where $m\in M$ and $n\in N$, subject to the following relations:
\begin{align}
mm'\cwedge n &= ({}^mm'\cwedge {}^mn)(m\cwedge n),\notag \\
m\cwedge nn' &= (m\cwedge n)({}^nm\cwedge {}^nn'), \label{eq:cwed}\\
x\cwedge y & = 1,\notag
\end{align}
for all $m,m'\in M$ $n,n'\in N$, and all $x\in M$ and $y\in N$ with $[x,y]=1$. If we denote $\M_0(M,N)=\langle m\wedge n\mid
m\in M,n\in N,[m,n]=1\rangle$, then we have that $M\cwedge N=(M\wedge N)/\M_0(M,N)$.

Let $L$ be a group. A function $\phi : M \times N \to L$ is called a {\it $\Bt_0$-pairing} if for all $m, m'\in M$, $n, n' \in N$,
and for all $x\in M$, $y\in N$ with $[x,y]=1$,
\begin{align*}
\phi (mm', n) &= \phi ({}^mm', {}^mn)\phi (m,n),\\
\phi (m,nn') &= \phi (m,n)\phi ({}^nm, {}^nn'),\\
\phi (x,y) &=1.
\end{align*}
Clearly a $\Bt_0$-pairing $\phi$ determines a unique homomorphism of groups 
$\phi ^*: M\cwedge N\to L$ such that $\phi ^*(m\cwedge n)=\phi (m,n)$
for all $m\in M$, $n\in N$. An example of a $\Bt_0$-pairing is the commutator map $M\times N\to [M,N]$. It induces 
a homomorphism $\tilde{\kappa}:M\cwedge N\to [M,N]$ such that $\tilde{\kappa}(m\cwedge n)=[m,n]$ for all $m\in M$ and
$n\in N$. We denote the kernel of this homomorphism by $\Bt_0(M,N)$.

In the case when $M=N=G$, we have that $\M_0(G,G)=\M_0(G)$ and $\Bt_0(G,G)=\Bt_0(G)$. We therefore
have a central extension
$$\xymatrix{ 0\ar[r] & \Bt_0(G)\ar[r] & G\cwedge G \ar[r]^{\tilde{\kappa}} & \gamma _2(G)\ar[r] & 1},$$
where $\tilde{\kappa}$ is the commutator map. Thus one can interpret $\Bt_0(G)$ as a measure of the extent to which
relations among commutators in $G$ fail to be consequences of `universal'
commutator relations given by the images of relations \eqref{eq:cwed} under the commutator map.

\begin{proposition}
\label{p:curlquot}
Let $M$ and $N$ be normal subgroups of a group $G$. Let $K\le M\cap N$ be a normal subgroup of $G$. Then
$M/K\cwedge N/K\cong (M\cwedge N)/J$, where
$J=\langle m\cwedge n\mid m\in M, n\in N, [m,n]\in K\rangle$.
\end{proposition}
\begin{proof}
The map $M/K\times N/K\to (M\cwedge N)/J$ given by $(mK,nK)\mapsto (m\cwedge n)J$ is well defined and
a $\Bt_0$-pairing, hence it induces a homomorphism $\varphi :M/K\cwedge N/K\to (M\cwedge N)/J$. On the other hand,
we have a canonical $\Bt_0$-pairing $M\times N\to M/K\cwedge N/K$ that induces a homomorphism $M\cwedge N\to M/K\cwedge N/K$.
Under this homomorphism $J$ gets mapped to 1, hence we have a homomorphism $\psi :(M\cwedge N)/J\to M/K\cwedge N/K$ whose
inverse is $\varphi$.
\end{proof}

Schur \cite{Sch07}, cf. also \cite[Kapitel V]{Hup67}, developed the theory of stem extensions. Here we indicate the role $\Bt_0(G)$ and $G\cwedge G$
within the theory.
Let $G$ be a finite group and denote by $(E,\pi ,A)$ the central extension
\begin{equation}
\label{eq:centext}
\xymatrix{ 1\ar[r] & A\ar[r] & E\ar[r]^\pi & G\ar[r] & 1}
\end{equation}
of $G$. If the transgression homomorphism $\tra :\Hom (A,\QZ )\to \HH^2(G,\QZ )$ is injective, then
we say that $(E,\pi ,A)$ is a {\it stem extension} of $G$, and that the group $B=\im\tra$ is {\it produced} by $(E,\pi ,A)$. One can show
that a central extension $(E,\pi ,A)$ of $G$ is a stem extension if and only if $A\le \gamma _2(E)$.
 In this case we have
that $B\cong A$.

By a well
known result of Schur \cite{Sch07}, every subgroup of $\HH^2(G,\QZ )$ is produced by some stem extension of $G$. This, in particular,
applies to $\B_0(G)$. In terms of its homological counterpart $\Bt_0(G)$, we obtain the following result.

\begin{theorem}
\label{t:stem}
Let $G$ be a finite group.
Let $(E,\pi ,A)$ be a stem extension that produces $\Bt_0(G)$. Then $|\gamma _2(E)|=|G\cwedge G|$. Furthermore, there exists a stem
extension $(E,\pi ,A)$ of $G$ producing $\Bt_0(G)$ such that $\gamma _2(E)\cong G\cwedge G$.
\end{theorem}
\begin{proof}
Let $G=F/R$ be a free presentation of $G$. By Proposition \ref{p:curlquot} we have that
$G\cwedge G\cong (F\cwedge F)/J$, where $J=\langle x\cwedge y\mid x,y\in F, [x,y]\in R\rangle$. Since the centralizer of every
nontrivial element of $F$ is cyclic, we have $F\cwedge F=F\wedge F$. As $\HH_2(F,\mathbb{Z})=0$, the commutator map
$\kappa:F\wedge F\to \gamma _2(F)$ is an isomorphism. From here we conclude that $G\cwedge G\cong
\gamma _2(F)/\langle \K(F)\cap R\rangle$.

Let $(E,\pi ,A)$ be a stem extension of $G=F/R$ producing $\Bt_0(G)$. We have that $A\cong \Bt_0(G)$.
Let $\{ x_1,\ldots ,x_n\}$ be the set of free generators of $F$. For every $1\le i\le n$
choose $e_i\in E$ such that $\pi (e_i)=x_iR$. As $A\le Z(E)\cap \gamma _2(E)$, we conclude that $A$ is contained in the Frattini
subgroup $\Frat (E)$ of $E$, cf. \cite[Satz III.3.12]{Hup67}. Thus $e_1,\ldots ,e_n$ generate $E$. From here it follows that there is an epimorphism
$\sigma :F\to E$ such that $\sigma (x_i)=e_i$ for all $i=1,\ldots ,n$. Denote $C=\ker\sigma$. It is straightforward to see that $C\le R$.
Since $\pi(\sigma(x))=xR$ for every $x\in F$, we have that $\sigma (R)=A$ and $\sigma ^{-1}(A)=R$. From here we obtain
$[R,F]\le C$. We claim that
$\sigma  (R\cap\gamma _2(F))=A$. For, if $a\in A=A\cap \gamma _2(E)=\sigma (R)\cap\sigma (\gamma _2(F))$, then we can write
$a=\sigma (r)=\sigma (\omega )$ for some $r\in R$ and $\omega\in\gamma _2(F)$. It follows that $\omega r^{-1}\in C\le R$, hence
$\omega\in R\cap\gamma _2(F)$, as required. If $\bar{\sigma}$ is the restriction of $\sigma$ to $R\cap\gamma _2(F)$, then
$\ker\bar{\sigma}=C\cap \gamma _2(F)$. Therefore we have
$(R\cap\gamma _2(F))/(C\cap\gamma _2(F))\cong A\cong (R\cap \gamma _2(F))/\langle R\cap \K(F)\rangle$. This in particular shows that
$|C\cap\gamma _2(F):[R,F]|=|\langle R\cap \K(F)\rangle :[R,F]|$. From here we obtain
$|\gamma _2(E)|=|\gamma _2(F):C\cap\gamma _2(F)|=|\gamma _2(F):[R,F]|/|C\cap\gamma _2(F):[R,F]|=
|\gamma _2(F):[R,F]|/|\langle R\cap \K(F)\rangle :[R,F]|=|\gamma _2(F):\langle R\cap \K(F)\rangle |=|G\cwedge G|$.

It remains to construct a stem extension $(E,\pi ,A)$ of $G=F/R$ producing $\Bt_0(G)$ such that $\gamma _2(E)\cong G\cwedge G$.
Denote $B=(R\cap \gamma _2(F))/\langle R\cap \K(F)\rangle$ and $T=R/\langle R\cap \K(F)\rangle$. Then $T/B\cong R/(R\cap\gamma _2(F))$
is free abelian, hence $B$ is complemented in $T$. Denote its complement by $\bar{C}=C/\langle R\cap \K(F)\rangle$, and put
$E=F/C$, $A=R/C$. Let $\pi :E\to G$ be the canonical epimorphism. Then $\ker\pi =A$. As $[R,F]\le \langle R\cap \K(F)\rangle\le C$,
it follows that $(E,\pi ,A)$ is a central extension of $G$. We have $A\cong T/\bar{C}=B\bar{C}/\bar{C}\cong C(R\cap\gamma _2(F))/C$,
therefore $A\le\gamma _2(E)$. This shows that $(E,\pi ,A)$ is a stem extension of $G$. As 
$\gamma _2(E)\cong \gamma _2(F)/(C\cap \gamma _2(F))=\gamma _2(F)/(C\cap (R\cap \gamma _2(F)))=\gamma _2(F)/\langle R\cap \K(F)\rangle
\cong G\cwedge G$, the assertion is proved.
\end{proof}

If a group $G$ is perfect, then $G\wedge G$ is the universal central extension of $G$. This can be deduced readily, cf. \cite[Theorem 5.7]{Mil71}.
A similar description can be obtained for $G\cwedge G$. We say that a central extension $(E,\pi ,A)$ of a group $G$
is {\it commutativity-preserving} ({\it CP}) if commuting elements of $G$ 
lift to commuting elements in $E$. A CP extension $(U,\phi ,A)$ of a group $G$ is said to be {\it CP-universal} if
for every CP extension $(E,\psi ,B)$ of  $G$ there exists a homomorphism $\chi :U\to E$ that factors through $G$, i.e., 
$\psi\chi =\phi$. It is straightforward to see that a group $G$ admits, up to isomorphism, at most one CP-universal central extension.

The following results have their direct counterparts in the theory of universal central extensions. 
The proofs follow along the lines of those of \cite[Chapter 5]{Mil71}.

\begin{proposition}
\label{p:CP1}
A CP extension $(U,\phi ,A)$ of a group $G$ is CP-universal if and only if $U$ is perfect, and every CP extension of $U$ splits.
\end{proposition}
\begin{proof}
Assume first that $U$ is perfect, and that every CP extension of $U$ splits. Let $(E,\psi ,B)$ be an arbitrary CP extension of $G$.
Form $U\times _GE=\{ (u,e)\in U\times E\mid \phi (u)=\psi (e)\}$, and let $\pi :U\times _GE\to U$ be the projection to the first factor. Then
$(U\times _GE,\pi,\ker\pi )$ is a central extension of $U$, obviously a CP one. Thus it splits and therefore the section $\sigma :U\to U\times _GE$
induces a homomorphism $\chi :U\to E$. Since $U$ is perfect, $\chi$ is uniquely determined \cite[Lemma 5.4]{Mil71}.

Conversely, suppose that $(U,\phi ,A)$ is a CP-universal central extension of $G$. Then $U$ is perfect by \cite[Lemma 5.5]{Mil71}. Let 
$(X,\psi ,B)$ be a CP extension of $U$. Then $(X,\phi\psi ,\ker\phi\psi )$ is a central extension of $G$. 
Take $x,y\in G$ with $[x,y]=1$. Since the extension $(U,\phi ,A)$ is CP, $x$ and $y$ have commuting lifts $x',y'\in U$ with respect to $\phi$.
The central extension $(X,\psi ,B)$ of $U$ is also CP, hence $x'$ and $y'$ have commuting lifts $x'',y''\in X$ with respect to $\psi$. This shows
that  $(X,\phi\psi ,\ker\phi\psi )$ is a CP extension of $G$. By the assumption, there exists a homomorphism $\chi :U\to X$ that factors
through $G$. We have that $\psi\chi$ is the identity map, hence the extension $(X,\psi ,B)$ of $G$ splits.
\end{proof}

\begin{proposition}
\label{p:CP2}
A group $G$ admits a CP-universal central extension if and only if it is perfect. In the latter case, 
$(G\cwedge G,\tilde{\kappa},\Bt_0(G))$ is the CP-universal
central extension of $G$.
\end{proposition}
\begin{proof}
Let $G$ be a perfect group. Suppose $G$ is given by the free presentation $G=F/R$, and denote $K=\langle \K(F)\cap R\rangle$.
We have a canonical surjection $\rho :F/K\to F/R$, and $\ker \rho =R/K$ is central $F/K$. By \cite[Lemma 5.6]{Mil71}, the group 
$\gamma _2(F)/K$, together with the appropriate restriction of $\rho$,
is a perfect central extension of $\gamma _2(G)=G$. Let $x$ and $y$ be commuting elements of $G$. Then there exist
$f_1,f_2\in\gamma _2(F)$ such that $x=f_1R$, $y=f_2R$, and $[f_1,f_2]\in \K (F)\cap R\subseteq K$. This shows that the above central
extension of $G$ is CP. We claim that it is also CP-universal. Let $(E,\psi ,A)$ be another CP extension of $G$. As $F$ is free, there
exists a homomorphism $\tau :F\to X$ such that $\psi\tau =\rho$. Take an arbitrary $[f_1,f_2]\in \K (F)\cap R$, where $f_1,f_2\in F$. Since
$\rho (f_1)$ and $\rho (f_2)$ commute, there exist commuting lifts $e_1,e_2\in E$ of these with respect to $\psi$. We can write
$\tau (f_i)=e_iz_i$ for some $z_i\in A\le Z(E)$, $i=1,2$. Then $\tau ([f_1,f_2])=[e_1z_1,e_2z_2]=1$, hence $\tau$ induces a homomorphism
$\chi :F/K\to E$. The restriction of $\chi$ to $\gamma _2(F)/K$ gives the required map. The second statement follows from the proof of
Theorem \ref{t:stem}.

The converse is obvious.
\end{proof}
\section{Nilpotent groups of class 2}
\label{s:class2}

\noindent
The first examples of finite $p$-groups $G$ with $\B_0(G)\neq 0$ were found within the groups that are nilpotent of class 2,
cf. \cite{Sal84,Bog88}. In this section we find a new description of $\B_0(G)$ for an arbitrary group $G$ of class 2. This is
achieved via the group $G\cwedge G$.

Let the group $G$ be nilpotent of class 2 and consider $G\cwedge G$. As $\gamma _2(G)\le Z(G)$, it follows that
$[x,y]\cwedge z=1$ for all $x,y,z\in G$. In particular, $G\cwedge G$ is an abelian group; in fact, it is easy to see
that even the group $G\wedge G$ is abelian. It also follows that
$$
{}^z(x\cwedge y) = {}^zx\cwedge {}^zy
= [z,x]x\cwedge {}^zy
= x\cwedge [z,y]y
= x\cwedge y,
$$
therefore $G$ acts trivially on $G\cwedge G$. Thus the defining relations \eqref{eq:cwed} of $G\cwedge G$ show that
the mapping $G\times G\to G\cwedge G$ defined by $(x,y)\mapsto x\cwedge y$ is bilinear. By the above argument, this map induces
a well defined bilinear mapping $G^{\rm ab}\times G^{\rm ab}\to G\cwedge G$ given by $(\bar{x},\bar{y})\mapsto x\cwedge y$,
where $\bar{x}=x\gamma _2(G)$ and $\bar{y}=y\gamma _2(G)$. This in turn induces a surjective group homomorphism
$\Psi :G^{\rm ab}\wedge G^{\rm ab}\to G\cwedge G$ given by $\bar{x}\wedge\bar{y}\mapsto x\cwedge y$. Similarly,
there is a well defined commutator map $G^{\rm ab}\times G^{\rm ab}\to \gamma _2(G)$ defined by $(\bar{x},\bar{y})\mapsto [x,y]$.
Since $G$ is of class $2$, the latter mapping is also bilinear, hence it induces a surjective homomorphism
$\Phi :G^{\rm ab}\wedge G^{\rm ab}\to\gamma _2(G)$. We have that $\Phi=\tilde{\kappa}\Psi$.

\begin{proposition}
\label{p:B0class2}
Let $G$ be a group of class 2. Then $\Bt_0(G)$ is isomorphic to $\ker\Phi /\ker\Psi$.
\end{proposition}
\begin{proof}
Clearly $\ker\Psi\le\ker\Phi$. Let $\tau$ be the restriction of $\Psi$ to $\ker\Phi$. Take any $k\in\ker\tilde{\kappa}$. There
exists $t\in G^{\rm ab}\wedge G^{\rm ab}$ such that $\Psi (t)=k$. Then $\Phi (t)=0$, hence $\tau$ maps $\ker\Phi$ onto
$\ker\tilde{\kappa}$. Besides, $\ker\tau =\ker\Psi$, thus $\ker\Phi /\ker\Psi\cong\ker\tilde{\kappa}\cong \Bt_0(G)$.
\end{proof}

\begin{remark}
\label{r:ten}
We also have group homomorphisms $\Psi _1:G^{\rm ab}\otimes G^{\rm ab}\to G\cwedge G$
and $\Phi _1:G^{\rm ab}\otimes G^{\rm ab}\to\gamma _2(G)$, defined similarly as above. It can
be shown that $\Bt_0(G)\cong \ker\Phi _1/\ker\Psi _1$.
\end{remark}

Bogomolov \cite{Bog88} found a rather detailed description of $\B_0(G)$ when $G$ is a $p$-group of class 2 such 
that $G^{\rm ab}$ is an elementary abelian $p$-group. Here we propose an alternative approach via the Blackburn-Evens
theory \cite{Bla79}. First note that both
$\gamma _2(G)$ and $G\cwedge G$ are elementary abelian $p$-groups. Denote $V=G^{\rm ab}$ and $W=\gamma _2(G)$. We
can consider $V$ and $W$ as vector spaces over $\mathbb{F}_p$. For $v_1,v_2\in V$ denote $(v_1,v_2)=[x_1,x_2]$ where
$v_i=x_i\gamma _2(G)$. This gives us a bilinear map $V\times V\to W$. Let $X_1$ be the subspace of $V\otimes W$ spanned by all
$v_1\otimes (v_2,v_3)+v_2\otimes (v_3,v_1)+v_3\otimes (v_1,v_2)$, where $v_i\in V$. Furthermore, define the map $f:V\to W$ by
$f(g\gamma _2(G))=g^p$, and let $X_2$ be the  subspace of $V\otimes W$ spanned by all $v\otimes f(v)$, where $v\in V$. Put
$X=X_1+X_2$.
Straightforward verification, cf. \cite{Bla79}, shows that the map $\sigma :V\wedge V\to (V\otimes W)/X$ given
by $\sigma (v_1\wedge v_2)=v_1\otimes f(v_2)+{p\choose 2}v_2\otimes (v_1,v_2)+X$ is well defined and
$\mathbb{F}_p$-linear.
As both $V\wedge V$ and $(V\otimes W)/X$ are elementary abelian $p$-groups, there exists an elementary
abelian $p$-group $M^*$ with $N\le M^*$
such that
\begin{equation}
\label{eq:ex}
N\cong (V\otimes W)/X \; \hbox{ and } \; M^*/N\cong V\wedge V.
\end{equation}

\begin{theorem}
\label{t:B0eltab}
Let $G$ be a finite group of class 2 such that $G^{\rm ab}$ is an elementary abelian $p$-group. 
Under the isomorphisms given by \eqref{eq:ex},
let
$M/N$ correspond to $\ker\Phi$ in $M^*/N$, and $M_0/N$ correspond to $\ker\Psi$ in
$M^*/N$. Then
$\Bt_0(G)\cong M/M_0$ and $\M_0(G)\cong M_0$.
\end{theorem}
\begin{proof}
By a result of Blackburn and Evens \cite{Bla79}, $M\cong \M(G)$. Proposition \ref{p:B0class2} implies that
$\Bt_0(G)\cong\ker\Phi /\ker\Psi\cong M/M_0$. Since both $M$ and $M_0$ are elementary
abelian, it follows from Theorem \ref{t:main} that $\M_0(G)\cong M_0$. This concludes the proof.
\end{proof}
\section{Split extensions and Frobenius groups}
\label{s:frobenius}

\noindent
In this section all the groups are finite.
Let $G=N\rtimes Q$ be a split extension of the group $N$ by $Q$. Then the Schur multiplier of $G$ can be described
by a result of Tahara \cite{Tah72}, see also \cite[p. 28]{Kar87}. We have that $\HH^2(G,\QZ )$ is naturally isomorphic to
$\HH^2(Q,\QZ)\oplus \bar{\HH}^2(G,\QZ )$, where $\bar{\HH}^2(G,\QZ )=\ker \res ^G_Q$. Moreover, $\bar{\HH}^2(G,\QZ )$ fits
into the following exact sequence:
\begin{multline}
\label{eq:splitseq}
0\rightarrow \HH^1(Q,\HH^1(N,\QZ ))\rightarrow \bar{\HH}^2(G,\QZ )\rightarrow
\HH^2(N,\QZ )^Q\\
\rightarrow \HH^2(Q,\HH^1(N,\QZ )).
\end{multline}

\noindent
A description in terms of the nonabelian exterior products is obtained as follows. The commutator map
$G\wedge N\to [N,G]$ is a homomorphism of groups. Denote its kernel by $\M(G,N)$. The group $\M(G,N)$
is said to be the {\it Schur multiplier of the pair} $(G,N)$. Ellis \cite{Ell98} proved that
$\M(G)\cong \M(G,N)\oplus \M(Q)$, and $\M(G,N)\cong \ker (\M(G)\to \M(Q))$. Here $\M(G,N)$ is embedded into $\M(G)$ via the
the restriction $\iota _1$ of the natural homomorphism 
$G\wedge N\to G\wedge G$, and the embedding $\iota _2:\M(Q)\hookrightarrow \M(G)$ is induced by the split surjection
$\xymatrix@=20pt{G\ar@<0.5ex>@{->>}[r] &Q\ar@<0.5ex>@{-->}[l]}$.

Our aim is to describe $\Bt_0(G)$ in the case when $G$ is a split extension of $N$ by $Q$. At first we define a subgroup
$\bar{\M}_0(G,N)$ of $G\wedge N$ by
$$\bar{\M}_0(G,N)=\langle (a\wedge m)^{-1}(b\wedge n)(n\wedge m)\mid a,b\in G, m,n\in N, [a,b]=1,{}^bnm={}^amn\rangle .$$
It is straightforward to verify that $\M_0(G,N)\le \bar{\M}_0(G,N)\le \M(G,N)$.

\begin{theorem}
\label{t:semidir}
Let $G=N\rtimes Q$. Then $\Bt_0(G)\cong \M(G,N)/\bar{\M}_0(G,N)\oplus \Bt_0(Q)$.
\end{theorem}
\begin{proof}
Let $x,y\in G$ commute. We can write $x=n_1^{-1}q_1$ and $y=n_2^{-1}q_2$ for some
$n_1,n_2\in N$ and $q_1,q_2\in Q$. From $[x,y]=1$ we obtain
${}^{n_1^{-1}}[q_1,n_2]\cdot {}^{n_1^{-1}n_2^{-1}}[q_1,q_2]\cdot [n_1^{-1},y]=1$. As $N\cap Q=1$, we conclude that
$[q_1,q_2]=1$. Therefore $q_1\wedge q_2$ as an element of $G\wedge G$ belongs to $\iota _2\M_0(Q)$, hence it is central in $G\wedge G$
and $G$ acts trivially upon it. Now we have
\begin{align*}
x\wedge y & ={}^{n_1^{-1}}(q_1\wedge n_2^{-1})\cdot {}^{n_1^{-1}n_2^{-1}}(q_1\wedge q_2)\cdot
(n_1^{-1}\wedge n_2^{-1})\cdot {}^{n_2^{-1}}(n_1^{-1}\wedge q_2)\\
&= {}^{n_1^{-1}n_2^{-1}}\left ( {}^{n_2}(q_1\wedge n_2^{-1})\cdot {}^{n_1n_2}(n_1^{-1}\wedge n_2^{-1})
\cdot {}^{[n_2,n_1]n_1}(n_1^{-1}\wedge q_2) \right )(q_1\wedge q_2)\\
&= {}^{n_1^{-1}n_2^{-1}}\left ( (q_1\wedge n_2)^{-1}(q_2\wedge n_1)(n_1\wedge n_2) \right ) (q_1\wedge q_2).
\end{align*}
From $[x,y]=[q_1,q_2]=1$ we obtain that $[n_2,q_1][q_2,n_1][n_1,n_2]=1$, which is equivalent to
${}^{q_2}n_1n_2={}^{q_1}n_2n_1$. It follows from here that
$(q_1\wedge n_2)^{-1}(q_2\wedge n_1)(n_1\wedge n_2)\in \iota _1\bar{\M}_0(G,N)$. This shows
that $\M_0(G)\le \iota _1 \bar{\M}_0(G,N)\oplus \iota _2\M_0(Q)$. Conversely, it is clear that 
$\iota _2\M_0(Q)\le \M_0(G)$. Take any generator $\omega =(a\wedge m)^{-1}(b\wedge n)(n\wedge m)$ of
$\bar{\M}_0(G,N)$. Then we have that $[a,b]=1$ and ${}^bnm={}^amn$. The above calculation shows
that ${}^{n^{-1}m^{-1}}\omega =(n^{-1}a\wedge m^{-1}b)(b\wedge a)$. By our assumptions we have
$[n^{-1}a,m^{-1}b]=1$, therefore $\omega\in \M_0(G)$. From here we can finally conclude that
$\M_0(G)=\iota _1 \bar{\M}_0(G,N)\oplus \iota _2\M_0(Q)$, and this proves the assertion.
\end{proof}

The structure of $\Bt_0(G)$ can further be refined when $G$ is a Frobenius group.
A {\it Frobenius group} \cite[p. 496]{Hup67}
is a transitive permutation group such that no non-trivial element fixes more than 
one point and some non-trivial element fixes a point.
The subgroup $Q$ of a Frobenius group $G$ fixing a point 
is called the {\it Frobenius complement}. 
By a theorem of Frobenius \cite[p. 496]{Hup67}, the set
$$N=G\setminus \bigcup _{g\in G} {}^g(Q\setminus\{ 1\})$$
is a normal subgroup in $G$ called the {\it Frobenius kernel} $N$. We have that $G=N\rtimes Q$, and $Q$ acts fixed-point-freely
upon $N$. We have that $Q\cap {}^gQ=1$ for every $g\in G\setminus Q$, and so if $\{ g_1,\ldots ,g_r\}$ is a left transversal
of $Q$ in $G$ then we have a {\it Frobenius partition}
\begin{equation}
\label{eq:part}
G=N\,\dot{\cup}\, {}^{g_1}Q\,\dot{\cup}\,\cdots\,\dot{\cup}\, {}^{g_r}Q,
\end{equation}
where the word `partition' means that the intersection of two different components is 1.

At first we describe the Schur multiplier of a Frobenius group by refining the above mentioned result of Tahara.

\begin{proposition}
\label{p:H2frob}
Let $G$ be a Frobenius group with Frobenius kernel $N$ and complement $Q$. Then
$\HH^2(G,\QZ )\cong \HH^2(N,\QZ )^Q\cong \M(G,N)$.
\end{proposition}
\begin{proof}
By Tahara's result we have
$\HH^2(G,\QZ )\cong \HH^2(Q,\QZ )\oplus \bar{\HH}^2(G,\QZ )$.
The Sylow $p$-subgroups of $Q$ are cyclic if $p$ is odd, and either cyclic or generalized quaternion groups if $p=2$ 
\cite[Hauptsatz V.8.7]{Hup67}.
Thus $\HH^2(P,\QZ )=0$ for every Sylow $p$-subgroup $P$ of $Q$ and every prime $p$ dividing the order of $Q$. It follows from here
that $\HH^2(Q,\QZ )=0$.
It remains to show that $\bar{\HH}^2(G,\QZ )\cong \HH^2(N,\QZ )^Q$. 
By \cite[Satz V.8.3]{Hup67} we have $\gcd (|N|,|Q|)=1$, which clearly implies $\HH^i((Q,\HH^1(N,\QZ ))=1$ for all $i\ge 1$,
hence the exact sequence \eqref{eq:splitseq} gives the isomorphism $\HH^2(G,\QZ )\cong \HH^2(N,\QZ )^Q$. The fact that
the latter is isomorphic to $\M(G,N)$ follows from the above mentioned result of Ellis.
\end{proof}

Moving on to $\Bt_0(G)$, where $G$ is a Frobenius group, we first need to describe the structure of commuting pairs in $G$.
In the Frobenius case, these are particularly well behaved, as the following result shows.

\begin{lemma}
\label{l:commfrob}
Let $G$ be a Frobenius group with the Frobenius kernel $N$ and complement $Q$. Let $x,y\in G$ commute.
Then either $x,y\in N$ or there exists $g\in G$ such that both $x$ and $y$ belong to ${}^gQ$.
\end{lemma}
\begin{proof}
Let $G$ have a Frobenius partition as given by \eqref{eq:part}.
Suppose that $[x,y]=1$ for $x,y\in G\setminus\{ 1\}$. We may further suppose that at least one of these elements does not belong to $N$.
Assume first that $x\in N$ and $y\notin N$. Without loss of generality we can write
$y={}^{g_1}q$ for some $q\in Q$. We have ${}^{xg_1}q={}^{g_1}q$, therefore
${}^{{}^{g_1^{-1}}x}q=q$. This can be rewritten as
${}^q\left ( {}^{g_1^{-1}}x\right )={}^{g_1^{-1}}x$. Since $Q$ acts fixed-point-freely on $N$, we conclude that
$q=1$ or $x=1$, a contradiction.

Assume now that $x$ and $y$ belong to different conjugates of $Q$. Without loss of generality we may assume that
$x\in Q$ and $y\in {}^{g_1}Q$ where $g_1\notin Q$. We can write $y={}^{g_1}q$, where $q\in Q$ and
$g_1=f_1q_1$ with $f_1\in F\setminus\{ 1\}$ and $q_1\in Q$. Denote $\tilde{q}={}^{q_1}q$. From ${}^xy=y$ we conclude that
${}^{f_1^{-1}xf_1}\tilde{q}=\tilde{q}\in Q\cap {}^{f_1^{-1}xf_1}Q$. As $\tilde{q}\neq 0$, we obtain that
$f_1^{-1}xf_1\in Q$, hence $x\in {}^{f_1}Q$. But $Q\cap {}^{f_1}Q=1$, and this is contrary to the assumption that $x\neq 1$.
This concludes the proof.
\end{proof}

\begin{corollary}
\label{c:frob1}
Let $G$ be a Frobenius group with the Frobenius kernel $N$. Then
$$\Bt_0(G)\cong \frac{\M(G,N)}{\im (\M_0(N)\to \M(G,N))}.$$
\end{corollary}
\begin{proof}
Denote $N_0=\im (\M_0(N)\to \M(G,N))$.
Let $x,y\in G$ and suppose that $[x,y]=1$.
By Lemma  \ref{l:commfrob} we either have that $x,y\in N$ or there exists $g\in G$ such that $x={}^gq_1$ and $y={}^gq_2$
for some $q_1,q_2\in Q$. We clearly have that $[q_1,q_2]=1$, hence $x\wedge y={}^gq_1\wedge{}^gq_2={}^g(q_1\wedge q_2)=
q_1\wedge q_2$. This shows that
$\M_0(G)=\langle x\wedge y\mid [x,y]=1, \hbox{either } (x,y)\in N\times N \hbox{ or } (x,y)\in Q\times Q\rangle$.
In view of the above notations we can thus write
$\M_0(G)=\iota _1 N_0\oplus \iota _2\M_0(Q)$. As $\HH^2(Q,\QZ )=0$, we have that $\M_0(Q)=0$, and hence the result.
\end{proof}

\begin{corollary}
\label{c:frob2}
Let $G$ be a Frobenius group with the Frobenius kernel $N$. Then
$$\B_0(G)=\bigcap _{A\in\mathcal{C}} \ker \res ^G_A,$$
where $\mathcal{C}$ is the family of all bicyclic subgroups of $N$.
\end{corollary}
\begin{proof}
Denote $B_0=\bigcap _{A\in\mathcal{C}} \ker \res ^G_A$. By a result of Bogomolov \cite{Bog88}
we have that $\B_0(G)=\bigcap _{A\in\mathcal{B}} \ker \res ^G_A$, where $\mathcal{B}$ is the collection of all
bicyclic subgroups of $G$, hence $\B_0(G)\le B_0$. Now let $\gamma\in B_0$. 
Fix an arbitrary
$B=\langle x,y\rangle\in\mathcal{B}$. If $x,y\in N$, then $B\in\mathcal{C}$, and thus $\res ^G_B\gamma =0$.
Otherwise, Lemma \ref{l:commfrob} implies that there exists $g\in G$ such that $x={}^gq_1$ and $y={}^gq_2$ for some
$q_1,q_2\in Q$. Clearly we have $[q_1,q_2]=1$. As $\HH^2({}^gQ,\QZ )=0$, we have 
$\HH^2(G,\QZ )=\ker\res ^G_{{}^gQ}\le\ker\res^G_B$, hence we again have $\res^G_B\gamma =0$. We conclude that $\gamma\in \B_0(G)$.
\end{proof}

\begin{corollary}
\label{c:abel}
Let $G$ be a Frobenius group with abelian Frobenius kernel. Then $\B_0(G)=0$.
\end{corollary}
\begin{proof}
Let $N$ be the Frobenius kernel of $G$ and $Q$ a complement of $N$ in $G$. As $N$ is abelian,
application of Corollary \ref{c:frob2} gives $\B_0(G)=\ker\res ^G_N$. Thus it suffices to show that the map
$\res ^G_N$ is injective. Let $\cor ^G_N:\HH^2(N,\QZ )\to \HH^2(G,\QZ )$ be the cohomological corestriction
map. 
Let $p$ be a prime dividing $|N|$, and denote the restriction of the map $\res ^G_N$ to the $p$-part
$\HH^2(G,\QZ )_p$ of $\HH^2(G,\QZ )$ by $\res ^G_N(p)$. Similarly, let
$\cor ^G_N(p)$ be the restriction of $\cor ^G_N$ to $\HH^2(N,\QZ )_p$.
Then $\cor ^G_N(p)\res ^G_N(p):\HH^2(G,\QZ )_p\to \HH^2(G,\QZ )_p$ is multiplication by $n=|G:N|=|Q|$.
As $p$ is coprime to $n$, it follows that $\res ^G_N(p)$ is injective for every $p$ dividing $|N|$. Therefore
$\res ^G_N$ is injective, as required.
\end{proof}
\section{The functor $\Bt_0$ in K-theory}
\label{s:ktheory}

\noindent
In this section, the role of the functor $\Bt_0$ within K-theory is outlined.
We first briefly recall some of the basic notions of K-theory. For unexplained notations and further account we refer to Milnor's book \cite{Mil71}.
Throughout this section let $\Lambda$ be a ring with 1. The group $\GL (\Lambda )$ is the direct limit of the chain
$\GL(1,\Lambda )\subset\GL (2,\Lambda )\subset\cdots$, where $\GL (n,\Lambda )$ is embedded in $\GL (n+1,\Lambda )$ via
$A\mapsto\begin{bmatrix} A & 0\\ 0 & 1\end{bmatrix}$. Denote by $\E (\Lambda )$ the subgroup of $\GL (\Lambda )$ generated
by all elementary matrices, and let $\St (\Lambda )$ be the Steinberg group. Then the $\K_1$ and $\K_2$ functors are given by 
$\K_1\Lambda =\GL (\Lambda )/\E (\Lambda )$ and
$\K_2\Lambda =\ker (\Phi :\St (\Lambda)\twoheadrightarrow \E (\Lambda ))$, respectively. 
It is known that $\K_2\Lambda$ is precisely the center of $\St (\Lambda )$, $\K_2\Lambda\cong\HH _2(\E (\Lambda ),\mathbb{Z})$, and
the sequence
$$1\longrightarrow \K_2\Lambda\longrightarrow\St (\Lambda )\longrightarrow \GL (\Lambda )\longrightarrow \K_1\Lambda\longrightarrow 1$$
is exact.

The fact that $\K_2\Lambda$ can be identified with $\HH _2(\E (\Lambda ),\mathbb{Z})$ suggests the following definition.
For a ring $\Lambda$ set $\Bt _0\Lambda =\Bt _0 (\E (\Lambda ))$. This clearly defines a covariant functor
 from $\mathbf{Ring}$ to $\mathbf{Ab}$. The group $\Bt_0\Lambda$ fits into the exact sequence
 $$1\longrightarrow \Bt_0\Lambda\longrightarrow \E (\Lambda )\cwedge \E (\Lambda )\longrightarrow \E (\Lambda )\longrightarrow 1.$$
 Thus $\Bt _0\Lambda$ can be considered a measure of the extent to which
relations among commutators in $\GL (\Lambda )$ fail to be consequences of `universal'
relations of $\E (\Lambda )\cwedge \E (\Lambda )$. Another description of $\Bt_0\Lambda$ can be obtained via
the Steinberg group. Denote $\M _0\Lambda =\langle \K (\St (\Lambda ))\cap\K _2\Lambda\rangle$. Then we have the following result.

\begin{theorem}
\label{t:stein}
Let $\Lambda$ be a ring. Then $\E (\Lambda )\cwedge \E (\Lambda )$ is naturally isomorphic to
$\St (\Lambda )/\M _0\Lambda$, and $\Bt _0\Lambda\cong \K_2\Lambda /\M_0\Lambda$.
\end{theorem}
\begin{proof}
The group $\St (\Lambda )$ is the universal central extension of $\E (\Lambda )$ \cite[Theorem 5.10]{Mil71}.
Since $\St (\Lambda )$ is perfect, it follows from \cite{Mil52} that $\St (\Lambda )\cong \E (\Lambda )\wedge \E (\Lambda )$. The isomorphism
$\psi : \E (\Lambda )\wedge \E (\Lambda )\to \St (\Lambda )$ can be chosen so that we have the following commutative diagram with exact rows:
$$
\xymatrix{
1\ar[r] & \M (\E (\Lambda )) \ar[r]\ar[d]_{\psi |_{ \M (\E (\Lambda ))}}^\cong & \E (\Lambda )\wedge \E (\Lambda ) \ar[r]^\kappa\ar[d]_{\psi}^\cong & 
\E (\Lambda )\ar[r]\ar@{=}[d] & 1\\
1\ar[r] & \K _2\Lambda \ar[r] & \St (\Lambda )\ar[r]^\Phi & \E (\Lambda )\ar[r] & 1
}
$$
From here we get that
$$\E (\Lambda )\cwedge \E (\Lambda )=(\E (\Lambda )\wedge \E (\Lambda ))/\M _0(\E (\Lambda ))\cong
\St (\Lambda )/\psi (\M _0(\E (\Lambda ))).$$
As $\psi (\M _0(\E (\Lambda )))=\langle [x,y]\mid x,y\in \St (\Lambda ),[\Phi (x),\Phi (y)]=1\rangle=\langle [x,y]\mid x,y\in \St (\Lambda ),
[x,y]\in\K _2\Lambda \rangle =\M_0\Lambda$, we get the result.
\end{proof}

Theorem \ref{t:stein} thus shows that $\Bt _0\Lambda$
is the obstruction to $\K_2\Lambda$ being generated by commutators.
Alternatively, let $A,B\in\E (\Lambda )$ commute, and choose $a,b\in\St (\Lambda )$ such that
$A=\Phi (a)$ and $B=\Phi (b)$. Define $A\star B=[a,b]\in \K_2\Lambda$ to be the {\it Milnor element}
induced by $A$ and $B$, cf. \cite[p. 63]{Mil71}. The following is then straightforward.

\begin{proposition}
Let $\Lambda$ be a ring. Then $\M_0\Lambda = \langle A\star B\mid A,B\in \E (\Lambda ), [A,B]=1\rangle$. Thus
$\Bt_0\Lambda =0$ if and only if $\K_2\Lambda$ is generated by Milnor's elements.
\end{proposition}

The question as to whether $\K_2\Lambda$ is generated by Milnor's elements for every ring $\Lambda$ was posed by Bass, cf. Problem 3
of \cite{Den73}. As the group $\E (\Lambda )$ is perfect, 
the problem is equivalent to the question whether or not every CP extension of $\E (\Lambda )$ is trivial, cf. Proposition \ref{p:CP2}.

Now let $\{ x_{ij}^\lambda\mid i,j\in\mathbb{N},\lambda\in\Lambda\}$
be the standard generating set of $\St (\Lambda )$. For $u\in\Lambda ^\times$ define $w_{ij}(u)=x_{ij}^ux_{ji}^{-u^{-1}}x_{ij}^u$
and $h_{ij}(u)=w_{ij}(u)w_{ij}(-1)$. For $u,v\in\Lambda ^\times$ with $uv=vu$ let $\{u ,v\}=[h_{ij}(u),h_{ij}(v)]$ be the
{\it Steinberg symbol}. It is known that $\K_2\mathbb{Z}$ is generated by the Steinberg symbol $\{ -1,-1\}$, cf. \cite[Corollary 10.2]{Mil71}. 
This, together
with Theorem \ref{t:stein}, implies that $\Bt_0\mathbb{Z}=0$. Similarly, we have the following.

\begin{corollary}
\label{c:semilocal}
Let $\Lambda$ be a commutative semilocal ring. Then $\Bt_0\Lambda =0$.
\end{corollary}
\begin{proof}
By a result of Stein and Dennis \cite[Theorem 2.7]{Ste73}, 
$\K_2\Lambda$ is generated by the Steinberg symbols $\{ u,v\}$, where $u,v\in\Lambda ^\times$,
hence the result follows from Theorem \ref{t:stein}.
\end{proof}

Our next goal is to compute $\Bt _0(\GL (\Lambda ))$ for an arbitrary ring $\Lambda$. Dennis \cite[Corollary 8]{Den76} showed
that $\HH _2(\GL (\Lambda ),\mathbb{Z})\cong \K_2\Lambda\oplus \HH_2(\GL (\Lambda )^{\rm ab},\mathbb{Z})$. 
The drawback is that the splitting is non-canonical. Instead
of that, we use a variant of the functor $\HH_2$ defined by Dennis. Given a group $G$, let $G\dwedge G$ be the group generated by symbols
$x\dwedge y$, where $x,y\in G$ are subject to the relations 
analogous to \eqref{eq:tens1} and \eqref{eq:tens2} in the definition of the nonabelian
exterior square $G\wedge G$ of the group $G$, and the relation \eqref{eq:tens3} is replaced by
\begin{equation}
\label{eq:tens3a}
(x\dwedge y)(y\dwedge x)=1
\end{equation}
for all $x,y\in G$. We clearly have the commutator homomorphism $\hat{\kappa} :G\dwedge G\to\gamma _2(G)$ given by 
$x\dwedge y\mapsto [x,y]$. Denote $\tilde{\HH}_2(G)=\ker\hat{\kappa}$. The latter group has a topological interpretation. Namely,
it follows from \cite{Bro87} that $\tilde{\HH}_2(G)\cong \pi _4(\Sigma ^2\K (G,1))$, where $\K (G,1)$ is the classifying space of $G$. 
Let $\tilde{\M}_0(G)=\langle x\dwedge y\mid x,y\in G, [x,y]=1\rangle$. Then the defining relations of $G\dwedge G$ imply
that $(G\dwedge G)/\tilde{\M}_0(G)\cong G\cwedge G$ and $\tilde{\HH}_2(G)/\tilde{\M}_0(G)\cong \Bt_0(G)$. If $G$ is perfect, then 
$(G\dwedge G,\hat{\kappa},\tilde{\HH}_2(G))$ is the universal central extension of $G$.

In our context it is crucial to note that there is a
canonical split exact sequence
\begin{equation}
\label{eq:HH2}
1\longrightarrow \tilde{\HH}_2(\E(\Lambda ))\longrightarrow \tilde{\HH}_2(\GL (\Lambda ))\longrightarrow
\tilde{\HH}_2(\GL (\Lambda )^{\rm ab})\longrightarrow 1,
\end{equation}
cf.  \cite[Theorem 7]{Den76}. This facilitates the proof of the following result:
\begin{theorem}
\label{t:b0gl}
Let $\Lambda$ be a ring. Then $\Bt _0\Lambda$ is naturally isomorphic to $\Bt _0(\GL (\Lambda ))$.
\end{theorem}
\begin{proof}
Let $G=\GL (\Lambda )$ and $E=\E (\Lambda )$. By \cite[Theorem 7]{Den76} we have that $G\dwedge G$ is naturally isomorphic
to $(E\dwedge E)\times (G^{\rm ab}\dwedge G^{\rm ab})$. Explicitly, there is a pairing 
$\star :G\times G\to E\dwedge E\cong \St (\Lambda )$ that extends the Milnor pairing defined above. This was found by Grayson, see \cite{Gra79}
for the details.
It turns out \cite[p. 27]{Gra79} that the map $\star$ 
preserves the relations \eqref{eq:tens1}, \eqref{eq:tens2} and \eqref{eq:tens3a}, hence it induces a well defined homomorphism
$\star :G\dwedge G\to E\dwedge E$. We have that $G^{\rm ab}\dwedge G^{\rm ab}=\tilde{\HH}_2(G^{\rm ab})$, and the pairing
$\circ :G\times G\to \tilde{\HH}_2(G^{\rm ab})$ given by $a\circ b=(a\oplus 1)\gamma_2(G)\cwedge (1\oplus b)\gamma _2(G)$ induces
a homomorphism $\circ :G\dwedge G\to \tilde{\HH}^2(G^{\rm ab})$. It can be proved \cite{Den76} that $x\dwedge y=(x\star y)(x\circ y)$
for every $x,y\in G$. From the definition of $\star$ it follows that if $x$ and $y$ commute, then $x\star y\in \K_2\Lambda$, and the elements
$x\circ y$ generate  $\tilde{\HH}^2(G^{\rm ab})$. Therefore $\tilde{\M}_0(G)=\tilde{M}_0(E)\times \tilde{\HH}^2(G^{\rm ab})$. This gives
the result.
\end{proof}

One of the fundamental results in K-theory is that
if $\Lambda$ is a ring and $I$ an ideal of $\Lambda$, then the sequence
\begin{equation}
\label{eq:KI}
\xymatrix{
\K_2(\Lambda ,I)\ar[r] & \K_2\Lambda\ar[r]^\tau & \K_2(\Lambda /I)\ar[r]^\partial &
\K_1(\Lambda ,I)\ar[r] & \K_1\Lambda\ar[r] & \cdots
}
\end{equation}
is exact \cite[Theorem 6.2]{Mil71}. In the rest of the section we derive a similar sequence for $\Bt_0$.
To this end, denote ${\rm J}(\Lambda ,I)=\partial (\M_0(\Lambda /I))$ and 
${\rm T}(\Lambda ,I)=\tau ^{-1}(\M_0(\Lambda /I))$. Then we have the following.

\begin{proposition}
\label{p:Kexact}
Let $\Lambda$ be a ring and $I$ an ideal of $\Lambda$. Then the sequence
$$\xymatrix{
1\ar[r] & \frac{{\rm T}(\Lambda ,I)}{\M_0\Lambda}\ar[r] & \Bt_0\Lambda\ar[r] & \Bt_0(\Lambda /I)\ar[r] &
\frac{\K_1(\Lambda ,I)}{{\rm J}(\Lambda ,I)} \ar[r] & \K_1\Lambda\ar[r] & \cdots
}
$$
is exact. 
\end{proposition}
\begin{proof}
The canonical homomorphism $\tau :\K_2\Lambda \to\K_2 (\Lambda /I)$ induces a homomorphism
$\tau ^\sharp :\Bt _0\Lambda\to\Bt_0 (\Lambda /I)$, whose kernel is precisely  ${\rm T}(\Lambda ,I)/\M_0\Lambda$.
By definition, the connecting homomorphism $\partial :\K_2(\Lambda /I)\to\K_1(\Lambda ,I)$ induces a natural map
$\partial^\sharp :\Bt_0(\Lambda /I)\to \K_1(\Lambda ,I)/{\rm J}(\Lambda ,I)$, and we have that 
$\ker\partial ^\sharp=\M_0(\Lambda /I)\ker\partial/\M_0(\Lambda /I)=\im\tau^\sharp$. Using the fact that the sequence
\eqref{eq:KI} is exact, we see that the canonical homomorphism $\sigma :\K_1(\Lambda ,I)\to\K_1\Lambda$ induces a
well defined homomorphism $\sigma ^\sharp :\K_1(\Lambda ,I)/{\rm J}(\Lambda ,I) \to \K_1\Lambda$. We have that 
$\im\sigma ^\sharp =\im\sigma$, and $\ker\sigma ^\sharp =\ker\sigma /{\rm J}(\Lambda ,I)=\im\partial /{\rm J}(\Lambda ,I)=\im\partial ^\sharp$.
This concludes the proof.
\end{proof}
\section{Computing $\Bt_0(G)$}
\label{s:comput}

\noindent 
A group $G$ is said to be {\it polycyclic} if it has a subnormal series
$1=G_0\triangleleft G_1\triangleleft\cdots\triangleleft G_n=G$ such that every factor $G_{i+1}/G_i$ is cyclic.
A finite group is polycyclic if and only if it is solvable. Computations with polycyclic groups are very efficient,
since several algorithmic problems are decidable within this class \cite{Sim94}.

Recently Eick and Nickel \cite{Eic08} developed efficient algorithms for computing 
nonabelian exterior squares and Schur multipliers of (possibly infinite) polycyclic groups.
Given a polycyclic group $G$, one can compute its nonabelian exterior square $G\wedge G$, the crossed pairing
$\lambda :G\times G\to G\wedge G$ given by $\lambda (x,y)=x\wedge y$, and the commutator map 
$\kappa:G\wedge G\to \gamma _2(G)$. The Schur multiplier $\HH_2(G,\mathbb{Z})$ is then computed as
$\M(G)=\ker\kappa$.

Let $G$ be a finite solvable group.
In order to compute $\Bt_0(G)$ it suffices to efficiently compute $\M_0(G)=\langle x\wedge y\mid x,y\in G, [x,y]=1\rangle$
as a subgroup of $\M(G)$. One would have to compute the set
$\mathcal{C}_G=\{ (x,y)\in G\times G\mid [x,y]=1\}$ of all commuting pairs in $G$ and then to compute $\M_0(G)$ as the group
generated by $\{\lambda (x,y)\mid (x,y)\in\mathcal{C}_G\}$. It turns out that this is computationally inefficient.
The first improvement is to observe that if $(x,y)\in\mathcal{C}_G$,
then also $({}^zx,{}^zy)\in\mathcal{C}_G$ for every $z\in G$. On the other hand, since $G$ acts trivially on $\M(G)$, we have that
${}^zx\wedge {}^zy={}^z(x\wedge y)=x\wedge y$,  
therefore it suffices to determine the conjugacy classes
$C_1,\ldots ,C_k$ and choose representatives $c_i\in C_i$, $i=1,\ldots ,k$. Then 
$\M_0(G)=\langle c_i\wedge x\mid c_i\in C_i, x\in C_G(c_i), i=1,\ldots ,k\rangle$. This can further be improved. For $x\in G$ consider the map
$\varphi :C_G(x)\to \ker\kappa$ given by
$y\mapsto x\wedge y$.
Let $y,z\in C_G(x)$. Then
$x\wedge yz=(x\wedge y)({}^yx\wedge {}^yz)=(x\wedge y)(x\wedge z)$, as $G$ acts trivially on $\ker\kappa$.
Thus $\varphi$ is a homomorphism. It follows from here that if $\mathcal{X}_i$ is a generating set of $C_G(c_i)$, $i=1,\ldots ,k$, then
$$\M_0(G)=\langle c_i\wedge x\mid c_i\in C_i, x\in \mathcal{X}_i, i=1,\ldots ,k\rangle .$$
This formula enables efficient computation of $\M_0(G)$, as it provides a reasonably small set of generators of this group. The
algorithm has been implemented in GAP \cite{GAP4}. 
It allows us to compute $\Bt_0(G)$ and $G\curlywedge G$ for finite solvable groups $G$.
A file of the GAP functions and commands for computing
$\Bt_0(G)$ can be found
at the author's website \cite{Mor10}.

Computer experiments reveal that there are no groups $G$ of order 32 with $\B_0(G)\neq 0$. This coincides with the hand calculations
done by Chu, Hu, Kang and Prokhorov \cite{Chu08}. Next, there are nine groups $G$ of order 64 with $\B_0(G)\neq 0$. If we denote
the $i$-th group in the GAP library of all groups of order $n$ by $G_n(i)$, then our computations using the above algorithm
show that we have $B(G_{64}(i))\neq 0$ for $i\in \{149,150,151,170,171,172,177,178,182\}$. In fact, in all these cases $\B_0(G_{64}(i))$ 
is isomorphic to $\Z/ 2\Z$. This confirms the calculations of Chu, Hu, Kang, and Kunyavski\u\i \ \cite[Theorem 10.8]{Chu09}.

Bogomolov \cite[Lemma 4.11]{Bog88} stated that if $G$ is a group with $G^{\rm ab}\cong \Z /p\Z\oplus \Z /p\Z$ and $\B_0(G)\neq 0$,
then $p>3$ and $|G|$ has to be at least $p^7$. His methods also imply that if $G$ is a $p$-group with $\B_0(G)\neq 0$, then $|G|\ge p^6$,
cf. \cite[Corollary 2.11]{Bog04}. On the other hand, our computations show that if $i\in\{ 28,29,30\}$, then 
$G_{243}(i)^{\rm ab}\cong  \Z /3\Z\oplus \Z /3\Z$ and $\B_0(G_{243}(i))\cong \Z /3\Z$. These can be double-checked by hand
calculations using the methods of \cite{Chu09}, and thus contradict both of the above Bogomolov's claims. We only sketch here
the relevant computations with the group $G_{243}(28)$.

\begin{example}
\label{e:243}
Denote $G=G_{243}(28)$. This group has the following polycyclic presentation:
\begin{multline*}
G=\langle g_1,g_2,g_3,g_4,g_5\mid g_1^3=1,\, g_2^3=g_4^2,\, g_3^3=g_5^2,\, g_4^3=1,\, g_5^3=1,\, [g_2,g_1]=g_3,\\
[g_3,g_1]=g_4,\, [g_3,g_2]=g_5,\, [g_4,g_1]=g_5,\, [g_i,g_j]=1\hbox{ for other } i>j\rangle .
\end{multline*}
Computations with GAP show that $G\wedge G$ is isomorphic to $G_{243}(34)$, and is generated by the set
$\{ g_2\wedge g_1,\, g_3\wedge g_1,\,
g_3\wedge g_2,\, g_4\wedge g_1\}$. Denote $w= (g_2\wedge g_3)(g_4\wedge g_1)$. We have that $|w|=9$, and
since $[g_2,g_3][g_4,g_1]=1$, it follows that $w\in\M (G)$. Further inspection of $G\wedge G$ reveals that 
$\M (G)=\langle w\rangle$ and $\M_0(G)\cong \langle w^3\rangle$, therefore
$\B_0(G)\cong \mathbb{Z}/3\mathbb{Z}$.
\end{example}

We have managed to find all solvable groups $G$ of order $\le 729$, apart from the orders 512, 576 and 640, with $\B_0(G)\neq 0$.
The numbers of such groups are given in Table \ref{tableB0}. As for the timings, it takes, for example, about seven seconds to compute
$\B_0(G)$ for a given group $G$ of order 729.
We note here that the algorithm works well even for reasonably larger solvable groups. For example, the free 2-generator Burnside group 
$B(2,4)$ of exponent 4 has order $2^{12}$, and our algorithm returns $\B_0(B(2,4))\cong \Z /2\Z$.

{\small
\begin{table}[h!tb]
\begin{tabular}{|c|c|c|}
\hline
$n$ & $\#$ of groups of order $n$ & $\#$ of $G$ with $\B_0(G)\neq 0$\\
\hline
64 & 267 & 9 \\
128 & 2328 & 230 \\
192 & 1543 & 54 \\
243 & 67 & 3 \\
256 &  56092 & 5953 \\
320 & 1640 & 54 \\
384 & 20169 & 1820\\
448 & 1396 & 54 \\
486 & 261 & 3 \\
704 & 1387 & 54 \\
729 & 504 & 85 \\
\hline
\end{tabular}
\caption{Numbers of groups $G$ with $\B_0(G)\neq 0$.}
\label{tableB0}
\end{table}

\end{document}